\documentclass[12 pt]{amsart}
\usepackage{amssymb, amsmath, amsfonts, amsthm, graphics}
\usepackage[hmargin=1 in, vmargin = 1 in]{geometry}
\usepackage{hyperref}
\usepackage{verbatim}
\usepackage[all]{xy}

\newcommand{\newword}[1]{\textbf{\emph{#1}}}

\newcommand{\PP}{\mathbb{P}}
\newcommand{\QQ}{\mathbb{Q}}
\newcommand{\RR}{\mathbb{R}}

\newcommand{\ZZ}{\mathbb{Z}}

\theoremstyle{definition}

\newtheorem*{dfn}{Definition}
\theoremstyle{definition}

\newtheorem{Theorem}{Theorem}
\newtheorem{Lemma}{Lemma}

\newtheorem{Reduction}{Reduction}
\newtheorem{Question}{Question}

\title{On the gonality, treewidth, and orientable genus of a graph}
\author{James Stankewicz}
\date{\today}

\begin{document}
\maketitle

\begin{abstract} We examine connections between the gonality, treewidth,
and orientable genus of a graph. Especially, we find that hyperelliptic
graphs in the sense of Baker and Norine are planar. We give a notion
of a bielliptic graph and show that each of these must embed into a closed
orientable surface of genus one. We also find, for all $g\ge 0$, trigonal graphs
of treewidth 3 and orientable genus $g$, and give analogues for graphs of
higher gonality.
\end{abstract}

The \newword{gonality}
of a graph can refer to many related notions inspired by the
Brill-Noether theory of an algebraic curve. Baker and Norine  were the first to
define it as the least degree of a non-constant
harmonic morphism of graphs $G \to T$ where
$G$ is the graph of interest and $T$ is a tree.
Compare this to the definition
of the gonality of an algebraic curve $C$ being the least degree of a
nonconstant morphism from $C$ to $\PP^1$. Several other notions of gonality
have been defined by other authors, including Caporaso \cite{Caporaso} and
Cornelisson-Kato-Kool \cite{CKK}. The last notion, \newword{stable gonality},
is notable because it
allows refinements formed by subdividing edges and adding leaves.
This does not change the \newword{orientable genus} of $G$,
or the least genus of a closed orientable surface into which $G$ embeds. This
stable gonality
is also notable as it admits a spectral lower bound, i.e., in terms
of the spectrum of the Laplacian of $G$. This is
particularly appealing because there is a certain type of graph which arises
from algebraic curves called \newword{Shimura curves}, and calculations
suggest that only finitely many are planar, while nearly every other invariant
of these graphs is spectral. Could it be that there is a connection
between stable gonality and orientable genus?
In the following we say that a graph is $d$\newword{-gonal} if its
stable gonality is $d$. In the case $d=2$, this ends up being equivalent to the
notion of a \newword{hyperelliptic graph} due to Baker and Norine when the
(Euler) genus of $G$ is at least 2
\cite[\S 5]{BNHypell}. 

\begin{Theorem}All hyperelliptic graphs are planar, and if $d\ge 3$ with
$d\not \equiv 2 \bmod 4$ then there exist
$3$-connected $d$-gonal graphs of all orientable genera at least $(d/2 -1)^2$.
\end{Theorem}

To make the above relation clear, recall that for a graph to be planar it is
equivalent to having orientable genus zero. Similarly, we say a graph is
\newword{toroidal} if its orientable genus is at most 1. This is not the end
of the story on the connection between gonality and orientable genus however,
as there is much more from the Brill-Noether theory of curves to be adapted to
the language of graphs. Consider for instance that an algebraic curve is called
bielliptic if it admits a degree 2 morphism to an algebraic curve of genus
one. Similarly, we let a graph $G$ be \newword{bielliptic} if it admits a
degree 2 harmonic morphism to a graph $G'$ of (Euler) genus one. We have the
following.

\begin{Theorem}\label{biellthm} All bielliptic graphs are toroidal.\end{Theorem}

Since the utility graph $K_{3,3}$ is bielliptic, this is the best that
could be hoped for. In fact, we are led to the following question.

\begin{Question} If $G$ is a graph which admits a degree 2 harmonic morphism
to a graph $G'$ of (Euler) genus $g$, is the orientable genus of $G$ at most
$g$?
\end{Question}

An affirmative answer to this question would not be totally optimal - e.g.,
 $K_5$ admits a degree 2 morphism to a genus 2 graph, but is toroidal. We know
of no counterexamples to this statement and the proof of Theorem
\ref{biellthm} suggests extensions but does not itself
 extend beyond the genus one case.

Much of this paper was developed in conversation with Spencer Backman. We thank
him for numerous ideas.

\section{Preliminaries on the involutions of graphs and hyperelliptic graphs}

Although the notion of a hyperelliptic graph is well-established,
we prefer to use the
equivalent definition furnished by the hyperelliptic involution \cite{BNHypell}.

\begin{dfn} A \newword{mixing involution} on a graph $G$ is an order-two
automorphism $\alpha: G \to G$ such that if $e$ is an edge between $x$ and
 $y$ fixed by $\alpha$ then $\alpha(x) = y$.
\end{dfn}

Note that a graph with a mixing involution $\alpha$ and without loops cannot
have any edges $e$ fixed by $\alpha$ between $\alpha$-fixed vertices $x$ and
 $y$. Note that if $G$ is a graph with loops, then the graph $G'$ obtained
by deleting those loops has the same orientable genus. We therefore make
our first reduction.

\begin{Reduction} Hereon, all graphs will be loopless.\end{Reduction}

We are now in the proper setting to consider
\newword{harmonic morphisms of graphs} \cite[\S 2.1]{BNHypell}, an example
 of which is given by the quotient of a graph $G$ by a mixing involution
$\alpha$. Notably, the quotient $G/\alpha$ has vertices of the form
 $\{v,\alpha(v)\}$ such that $v$ is a vertex of $G$ and edges of the form
 $\{e,\alpha(e)\}$ such that the bounding vertices of $e$ are inequivalent
under $\alpha$. One defines a map $G \to G/\alpha$ by sending vertices to
the obvious place, edges to the obvious place provided that their bounding
vertices are non-equivalent under $\alpha$. If $e$ is an edge of the form
$e(v, \alpha(v))$ then of course we must send $e$ to the quotient vertex
$\{v,\alpha(v)\}$.

In the terminology of Baker-Norine, if $G$ has at least 3 vertices, this map
is a harmonic morphism of degree 2.
All such morphisms on graphs with at least 3 vertices arise this way
\cite[Lemma 5.6]{BNHypell}. If $G$ has two vertices, then there is an obvious
mixing involution and the quotient is a point, and thus a tree,
and it is only because that map is constant that we do not say it has degree 2.

\begin{dfn}We say that a connected graph $G$ admitting a
mixing involution $\iota:G\to G$ such that $G/\iota$ is a tree is
\newword{hyperelliptic} and that $\iota$ is the corresponding
\newword{hyperelliptic involution}.\end{dfn}

This is a slightly nonstandard definition in that we don't require the genus
to be at least 2. Typically one stipulates that because when $G$
is 2-edge-connected and has genus $\ge 2$, such an involution must be unique
\cite[Corollary 5.15]{BNHypell}. Thankfully we can reduce to the
2-edge-connected case without pain by contracting all its bridges
 \cite[Corollary 5.11]{BNHypell}. There are no 2-edge connected trees, and the
only 2-edge connected genus one graphs are cycles, which are planar.

\begin{Reduction} Hereon, when we refer to the graph $G$, we will mean it to be
2-edge connected.
\end{Reduction}

Note also that a graph with all its bridges contracted has the same orientable
genus as the original graph. Of course we will allow other graphs to not be
2-edge connected. Indeed $G/\iota$ will often be a tree in what follows.
Before proceeding further, we review some examples.

\section{Hyperelliptic graphs associated to Shimura curves}

The literature on Shimura curves which is relevant to the task at hand is
simply too large and too technical to introduce here in a meaningful way. Let
it suffice to say that Ogg has determined all Shimura curves $X^D$ which are
hyperelliptic over $\overline\QQ$ \cite{Ogg}. In particular, note that in each
case $D$ is the product of two primes and so there are only two primes of bad
reduction to explore. In each case, the dual graph is also hyperelliptic.
The following code verifies that all of these dual graphs are planar.

\begin{verbatim}
Dlist := [26,35,38,39,51,55,57,58,62,69,74,
82,86,87,93,94,95,111,119,134,146,159,194,206];
// Ogg's list of Shimura curves hyperelliptic over QQbar
del := function(x)
if x eq 0 then return 0;
else return 1;
end if;
end function;
ReducedDualGraph := function(p,q)
// Returns in magma format the dual graph of X^{pq} over FFpbar
// Rather, the "reduced dual graph" with parallel edges collapsed
M := BrandtModule(q,1);
d := Dimension(M);
Mx := MatrixRing(Integers(),d);
Bx := Mx!HeckeOperator(M,p);
for i in [1..d] do for j in [1..d] do
Bx[i,j] := del(Bx[i,j]);
end for; end for;
return Graph<2*Dimension(M)|BlockMatrix(2,2,[[Mx!0,Bx],[Bx,Mx!0]])>;
end function;
for D in Dlist do
G1 := ReducedDualGraph(PrimeDivisors(D)[1],PrimeDivisors(D)[2]);
G2 := ReducedDualGraph(PrimeDivisors(D)[2],PrimeDivisors(D)[1]);
D,IsPlanar(G1),IsPlanar(G2);
end for;
\end{verbatim}

Similar lists exist for, e.g., bielliptic Shimura curves, each of which has
$D\le 546$. Similar code to the above suggests that if $X^D$ has a dual graph
(of its reduction modulo $p$ for $p\mid D$) which is planar and has at least
six vertices, then for $D\ge 500$ the complete list of $(D,p)$ is

$$
\begin{array}{r|l}
p & D \\ \hline
2 & 510,546 \\
3 & 510,570,690 \\
5 & 690,910,1110 \\
7 & 798,910 \\
11 & 1122 \\
13 & 1365\\
29 & 667,2958.
\end{array}
$$

\section{Planarity and Toroidality of graphs with involutions}

Suppose now $G$ is a graph which is 2-edge-connected, loopless, and has a
hyperelliptic involution $\iota$. A given vertex can be either fixed or moved
by $\iota$. We let $F$ denote the set of vertices which are fixed by $\iota$.
By definition, all other vertices are permuted, and there must be an even
number of these. Let $A$ and $B$ be disjoint sets of permuted vertices:
we let $a_1, \ldots, a_n$ be
the elements of $A$, so $B = \{b_1 = \iota(a_1), \ldots, b_n = \iota(a_n)\}$.

The edges of $G$ must therefore fall into one of the following categories.

\begin{itemize}
\item The set $E_A$ of edges from $A$ to itself.
\item The set $E_B = \iota(E_A)$ of edges from $B$ to itself.
\item The set $E_F$ of edges from $F$ to itself.
\item The ``horizontal edges'' $H$ from some $a_i$ to $b_i$.
\item The ``cross edges'' $C$ from some $a_i$ to some $b_j$ such that $i\ne j$.
\item The ``transfer edges'' $T_A$ and $T_B$ respectively from $F$ to $A$
and $F$ to $B$. Note that $T_B = \iota(T_A)$.
\end{itemize}

We note some properties of subgraphs of $G$.

\begin{Lemma}
The involution $\iota$ maps the subgraph $(A,E_A)$ isomorphically onto $(B,E_B)$
and both are a finite disjoint union of trees.
\end{Lemma}

\begin{proof}
The isomorphism between the two is simply given by restricting $\iota$ to
$(A,E_A)$. We must therefore have an isomorphic copy of $(A,E_A)$ in the
quotient $G/\iota$, which is a finite connected tree. Any subgraph of a tree
must be a disjoint union of trees and so the result follows.
\end{proof}

\begin{Lemma} The connected components of the subgraph $(F,E_F)$ are either
single vertices or chains of vertices $f_1, \ldots, f_r$ such that between
$f_i$ and $f_{i+1}$ there are exactly two edges and between $f_i$ and $f_j$
there are no edges if $|i-j|>1$.
\end{Lemma}

\begin{proof} Let $e\in E_F$ and let $f,f'$ be the bounding vertices of $e$.
Since $\iota$ fixes $f,f'$ and $\iota$ is mixing, we must have $\iota(e) \ne e$.
Therefore there are at least 2 edges between $f$ and $f'$. If we
suppose to the contrary that there was a
third edge $e'$ then $\iota(e')$ would be distinct from $e'$ again by the
mixing property. But also since $e'\ne e$ and $e' \ne \iota(e)$ we must also
have $\iota(e') \ne e$ and $e' \ne \iota(e)$. The quotient graph $G/\iota$
would then have a cycle $ee'$ and since the hyperelliptic involution is unique
we have a contradiction.

Therefore between any two vertices $f,f'$ in our subgraph $(F,E_F)$ there are
either zero or two edges. If $f,f',f''$ each have two edges between them, then
in the quotient, we would have a cycle $e(f,f')e(f',f'')e(f'',f)$. The result
follows.
\end{proof}

We see therefore that $(F,E_F)$ is planar, and although a given connected
component may have a cycle, and for the purpose of orientable genus we
may think of each one as a point. We can therefore make the following
reduction by replacing $F$ with the set of connected components of $F$ and
$E_F$ by the empty set.

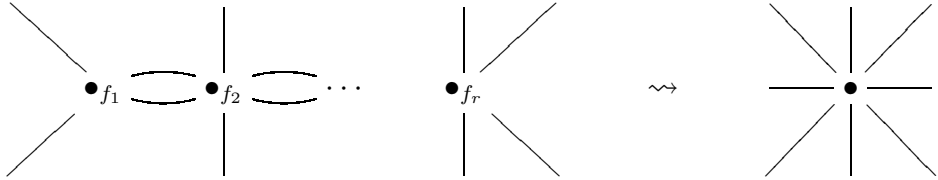
\[
\xymatrix{
&&&&&&&&&\\
&\bullet_{f_1}\ar@/^/@{-}[r]\ar@/_/@{-}[r]\ar@{-}[lu]\ar@{-}[ld] &
\bullet_{f_2}\ar@/^/@{-}[r]\ar@/_/@{-}[r]\ar@{-}[u]\ar@{-}[d] &
\cdots &
\bullet_{f_r}\ar@{-}[u]\ar@{-}[d]\ar@{-}[ru]\ar@{-}[rd]&&
\leadsto &&
\bullet \ar@{-}[u]\ar@{-}[ur]\ar@{-}[ul]\ar@{-}[l]\ar@{-}[r]\ar@{-}[ld]\ar@{-}[dr]\ar@{-}[d]&\\
&&&&&&&&&\\
}
\]

\begin{Reduction} We will assume $E_F$ is empty.\end{Reduction}

In the same way, we can replace $(A,E_A)$ and $(B,E_B)$ by the connected
components of each.

\begin{Reduction} We will assume $E_A$ and $E_B$ are empty.\end{Reduction}

Note that if we were to refine $G$ by adding a point in the middle of each
horizontal edge we would obtain a new graph. Embedding
this refined graph into an
orientable surface of genus $g$ induces an embedding of $G$
into the same surface.
We therefore refine $G$ by
adding a new element each of $F$, $T_A$ and $T_B$ as we eliminate $H$.

\begin{Reduction} We assume that $G$ has no horizontal edges.\end{Reduction}

We are now ready to prove our main theorem on hyperelliptic graphs.
If we can embed any connected graph $G$ into the plane,
then by adding a point at infinity, we give an embedding of this graph into
the 2-sphere $S^2$, and in fact that graph defines a CW-decomposition of $S^2$.
For instance, if $G$ has genus $g$ then this decomposition has $V(G)$ vertices,
$E(G)$ edges, and $g+1$ faces. By spherical inversion we can simply assume that
any one pair
$\{a_j,b_j\}$ lies on the same face as $\infty$, or that they lie
on the ``outside face.'' We will freely perform this in the following.

\begin{Theorem}\label{MainThm}
All hyperelliptic graphs are planar. Moreover there is an
embedding $\rho_G$ into $\RR^2$ under which any pair $\{a_j,b_j\}$ exchanged
by the hyperelliptic involution $\iota$ lie on a common face.
\end{Theorem}

\begin{proof} We induct on the size of $\# A = \# B$.
The following will be our inductive assumption.

\begin{itemize}
\item {\bf Ind}(n): All connected hyperelliptic graphs with
$\#A = \#B \le n$ admit a piecewise smooth (considering $G$
e.g., as a simplicial complex)
embedding $\rho_G: G \to \RR^2$ such that
\begin{enumerate}
\item If $\rho(v) = (x,y)$ then $\rho(\iota(v)) = (-x,y)$ and
\item If $\{a_i,b_i\}$ are exchanged by $\iota$ then there is a face $F$
of the CW decomposition of $S^2$ induced by $\rho_G$ such that
$a_i,b_i\in\partial F$.
\end{enumerate}
\end{itemize}

Clearly ${\bf Ind}(0)$ holds as we have shown that a hyperelliptic
graph which fixes each vertex is planar. Almost-as-clearly, ${\bf Ind}(1)$
holds because there are no crossing edges, and so all edges are transfer edges
by our reductive step. Since $G$ is connected, between each fixed point $f$
there is at least one transfer edge between $f$ and $a_1$ as well as $f$ and
$b_1$. There is also at most one such edge, because if there were two edges
between $f$ and $a_1$ then there would be a cycle in the quotient. It follows
that after our reductions, $G$ embeds into the plane as the banana graph
with midpoints. Of course, both $a_1$ and $b_1$ lie on the outside face.

Now suppose that $G$ has $\#A = n$ and ${\bf Ind}(n-1)$ is satsified. We let
\begin{itemize}
\item $A(n-1) = \{ a_1, \ldots, a_{n-1}\}$ and $B(n) = \iota(A(n-1))$
\item $T_A(n-1) = \{$ edges from $F$ to $A(n-1)\}$ and
$T_B(n-1) = \iota(T_A(n-1))$.
\item $C(n-1) = \{$ cross edges from $A(n-1)$ to $B(n-1)\}$.
\end{itemize}

We therefore let $G(n-1)$ be the graph whose vertices are
$A(n-1) \cup B(n-1) \cup F$
and whose edges are $T_A(n-1) \cup T_B(n-1) \cup C(n-1)$. As
$G(n-1)/\iota$ is a
subgraph of $G/\iota$, it is a finite disjoint union of trees. Let
$\Gamma_1,\ldots, \Gamma_m$ be the horizontal connected components of
$G(n-1)$, i.e. $\Gamma_i$ is either connected or the union of two
vertices exchanged by $\iota$. All of the
images of the $\Gamma_i$ in the quotient are connected trees. We note that
the connected
$\Gamma_i$ are hyperelliptic and so satisfy the conclusions
of ${\bf Ind}(n-1)$. 

Since $G$ is connected, for each $i$ there is a pair of transfer edges or a
pair of cross edges from $\{a_{n},b_{n}\}$ to $\Gamma_i$. In fact, there can be
either a cross edge $c_i$ from $a_n$ to some $b_k$ in $\Gamma_i$ or a transfer
edge $t_i$ from $a_n$ to a fixed point $f$ in $\Gamma_i$ and not both.
There cannot be more than
one else there would be a cycle in the quotient.

We therefore create a function $\psi_G:\{1,\ldots, m\} \to \{0,1\}$ where
$\psi(i)$ is $0$ if there is a transfer edge $a_{n}$ to $\Gamma_i$ and $1$
in the case of a cross edge. We roughly create $\rho_G$ as follows:
${\bf Ind}(n-1)$ gives us an embedding of each $\Gamma_i$ into $\RR^2$, but
moreover we can scale down into $[-1,1]^2$ and still be symmetric under $\iota$.
We stack each copy of $[-1,1]^2$ vertically in $\RR^2$, put $a_n$ to the left
of this column, $b_n$ to the right, and either directly attach the transfer
edge if $\psi_G(i)=0$ or possibly
first apply $\iota$ to $\Gamma_i$ before attaching the
cross edge if $\psi_G(i) =1$. Hidden in this is that if $\psi_G(i)=0$ we need
to make sure to perform spherical inversion to make sure that the fixed point
$f$ is on the outside face, and if $\psi_G(i)=1$ we need to make sure that both
$a_k$ and $b_k$ are on the outside face. This latter part explains the second
condition of
 ${\bf Ind}(n)$ and the remainder of the proof is simply verifying the
conditions of ${\bf Ind}(n)$ and making the construction explicit.

As noted, if $\psi_G(i) = 0$ then we may assume that our $\rho_{\Gamma_i}$ 
has $f$ on the outside
face. By scaling and shifting up or down we may assume that $\rho_{\Gamma_i}$
has image in the interior of $[-1,1]^2$ which is symmetric about the $y$-axis
 and $\rho_{\Gamma_i}(f) = (0,0)$. We may therefore draw
a symmetric pair of edges between $(0,0)$ and $(\pm 1, 0)$ which do not
intersect $\Gamma_i$. Note that these two new edges split the outside face of
$[-1,1]^2$ into two, but that $\Gamma_i$ lies entirely on one side of that
divide, so adding these edges does not change whether ${\bf Ind}(n)$
is satisfied. If $\psi_G$ is identically zero,
we embed a refinement of $G$ into
$\RR^2$ as follows: send $a_{n}, b_{n}$ to $(\pm 1, 0)$, use $\rho_{\Gamma_i}$
to send $\Gamma_i$ to $\{(x,y): -1\le x\le 1, i-1\le y\le i+1\}$. We can
symmetrically draw edges between $(\pm 1,0)$ and $(\pm 1, i)$ which are pairwise
disjoint and this produces an embedding $\rho_G$ which is symmetric under
$\iota$ and preserves the face condition of our inductive
assumption for $G$.
 
Now let's assume there are some $i$ such that $\psi_G(i) =1$. We assume that
$\Gamma_i$ is connected, else it is the disjoint union of two vertices, and
adding some cross edges does not change the face condition of ${\bf Ind}(n)$.
Let $\rho_i = \rho_{\Gamma_i}$ be an embedding so that $\rho_i(a_k),\rho_i(b_k)$
lie on the outside face with respectively positive and negative $x$-values,
and let $d_i$, $d_i'$ respectively be paths $(-1,0)$ to
$\rho_i(b_k)$ and $(1,0)$ to $\rho_i(a_k)$ such that $d_i' = \iota(d_i)$.  Could
it be that $d_i, d_i'$ put $a_j$ and $b_j$ on different faces?
\begin{itemize}
\item By
${\bf Ind}(n-1)$, $\rho_i(a_j)$ and $\rho_i(b_j)$ share a face,
and we need only worry
if it is the outside face.
\item If $a_j = a_k$
then $a_j$ still lies on the same face as $b_j$ even after adding $d_i$ and
$d_i'$.
\end{itemize}

So we assume $a_j$ and $b_j$ lie on the outside face, let $\gamma_j^+$
and $\gamma_j^-$ be smooth symmetric paths from $\rho_i(a_j)$ to $\rho_i(b_j)$
which lie above and below $\rho_i(\Gamma_i)$, meeting only at $\rho_i(a_j)$ and
$\rho_i(b_j)$. As such, $\gamma_j^+\cup \gamma_j^-$ forms a simple
Jordan curve, which
has an inside and outside defined by the mod 2 intersection number
\cite[\S 3.3]{GPTop}. Since $a_j\ne a_k$, the path $d_i$ has an odd
number of transverse
intersection points with $\gamma_j^+\cup \gamma_j^-$
up to multiplicity. If there is just
one, we are done, as it has to lie on precisely one of $\gamma_j^+$ and
$\gamma_j^-$. The non-intersecting path lies within the face we desire.
If there are three or more, we may pick an $\varepsilon>0$ less
than the distance from $\rho_i(\Gamma_i)$ to any of the points of
$d_i\cap(\gamma_j^+\cup\gamma_j^-)$. There is thus a smooth path between
$\rho_i(b_j)$ and $(-1,0)$ which agrees with $d_i$ at distance less than
$\varepsilon$ from $\rho_i(\Gamma_i)$, which is homotopic to $d_i$, and which
has precisely one point of intersection with $\gamma_j^+ \cup \gamma_j^-$. By
replacing $d_i$ with this path and $d_i'$ by the image under $\iota$ we have
reduced to the previous case. We conclude that
${\bf Ind}(n)$ holds and the proof
of our Theorem is complete.\end{proof}

For good measure, we give a second proof of the planarity of hyperelliptic
graphs.

\begin{proof}
By work of de Bruyn and Gijswijt \cite{Treewidth},
we know that for all graphs $G$, the stable
gonality of $G$ is bounded below by the treewidth of $G$. We know that $G$
is hyperelliptic if and only if the stable gonality is 2.
Since $G$ is hyperelliptic, we find that it has treewidth
2, and therefore is a subgraph of a series-parallel graph \cite{SP},
and is therefore planar.\end{proof}

There is also a third proof of this result due to Spencer Backman which
characterizes the ear decomposition of a hyperelliptic graph and which
predates work of de Bruyn and Gijswijt but was not written up.
While it may not seem so, these proofs work out to being very similar. Since
$G$ is hyperelliptic, $G/\iota$ is a tree. We may think of the inductive proof
as rooting that tree and thus
producing an embedding into a series-parallel graph. Note that our embedding
$\rho_G$ gives $G/\iota$ as $\rho_G(G)\cap\{(x,y): x\le 0\}$, so the source and
sink vertices are respectively $a_n$ and $b_n$.
The advantage of
working so explicitly is that some natural improvements present themselves.

\begin{Lemma} On any hyperelliptic graph $G$ with two pairs of vertices
$a_i\ne b_i$ and $a_j\ne b_j$ exchanged by the hyperelliptic involution, we can
find an embedding $\rho_{i,j}$ of $G$ into $\RR^2$ such that $a_i,b_i,a_j,b_j$
all lie on the boundary of a face. Moreover, the same is true when $a_i$ and
$b_i$ are replaced by a hyperelliptic fixed vertex.
\end{Lemma}

\begin{proof}
We proceed by induction in the same way as in the proof of
Theorem \ref{MainThm}. In fact,
if $a_i = a_j$ then our Lemma holds by appealing to Theorem \ref{MainThm}.
Therefore we
suppose that $a_i \ne a_j$ and thus $b_i\ne b_j$. We make all necessary
reductions to retain the notation
of $V(G) = A \cup B \cup F$ and $E(G) = C \cup T$. We know therefore that
$\#A = \# B \ge 2$. In the case of equality, $G$ is outerplanar.
If we do not have equality, we reorder $A$ and
$B$ so that $j = n$ and let $\Gamma_1,\ldots,\Gamma_m$ be the horizontal
connected components of $G(n-1)$ as in the proof of the Theorem.

Let $r$ be
such that $a_i\in \Gamma_r$ and let $k$ be such that there is a cross
edge from $a_n$ to $b_k$.
We apply our inductive hypothesis to $\Gamma_r$ to find an embedding of
$\Gamma_r$ into $\RR^2$ such that $a_i$ and $a_k$ share a face.
We use spherical
inversion to move that face to the outside, and thereby give an embedding
of $G$ into $\RR^2$ such that $a_i$ and $a_n = a_j$ share a face.

If $a_i$ and $b_i$ are replaced by a fixed vertex $f$, then we let $r$ be such
that $f\in\Gamma_r$ and we use spherical inversion to find a planar embedding
of $\Gamma_r$ such that $f$ lies on the outside face. The result follows in the
same way.
\end{proof}

With the above in mind, we recall that a bielliptic graph is one which admits a
mixing involution $\alpha$ such that $G/\alpha$ has genus one. We therefore
have the following.

\begin{Theorem}Bielliptic graphs are toroidal.
\end{Theorem}

\begin{proof}
Without loss of generality, we assume $G$ is 2-edge connected , and that the
genus of $G$ is at least 3, else $G$ is already planar.

Since $G/\alpha$ has genus one, there is an edge $\bar e$ of $G/\alpha$ such
that $G/\alpha - \bar e$ is a tree. Let $e,e'=\alpha(e)$
be the preimages of $\bar e$
in $G$ and let $G_0 = G - \{e,e'\}$ with $\alpha_0$ the induced involution,
whose quotient is $G/\alpha - \bar e$.

\[
\begin{array}{ccc}
G_0 & \hookrightarrow & G \\
\downarrow & & \downarrow \\
G_0/\alpha_0 & \hookrightarrow & G/\alpha
\end{array}
\]

We show first that $G_0$ is connected: if not,
let $a,b$ be the endpoints of $e$ and $G_a, G_b$ the connected components of
each in $G_0$. In which of these can we find $\alpha_0(a)$ and $\alpha_0(b)$?
If there is a path $\gamma_a$
between $a$ and $\alpha_0(a)$ then $G_0$ is
connected, as there is a unique simple path in $G_0/\alpha_0$ between
$\alpha_0^\sim(v_1)$ and $\alpha_0^\sim(v_2)$ for any
$v_1\in G_a$ and $v_2\in G_b$. This path lifts to a path $\gamma$ between either
 $v_1$ and $v_2$ (in which case $G_0$ is connected)
or $v_1$ and $\alpha_0(v_2)$ (in
which case $\gamma_a\alpha_0(\gamma)$ is a path between $v_1$ and $v_2$). Thus
there is no such path $\gamma_a$ when $G_0$ is disconnected. In other words,
when $G_0$ is disconnected, $\alpha_0(a)\not\in G_a$. Since $\alpha_0$ is
an isomorphism, it must exchange $G_a$ with $G_b$ so that
$\alpha_0: G_a\stackrel{\sim}{\to} G_b$.
But then the quotient is a tree, so $G_a$ and $G_b$ are trees.
This however contradicts the statement that the genus of $G$ is at least 3.

It follows then that $G_0$ is hyperelliptic, and therefore planar. Moreover the
embedding is planar in such a way as to recognize $\alpha_0$ as reflection
about the $y$-axis.
Let $a,b$ be the endpoints of $e$ and $a',b'$ be the endpoints of $e'$,
so moreover
we can find a planar embedding of $G_0$ such that $a,a',b,b'$ all lie on the
outside face. The boundary of this outside face is a Jordan curve containing
$a,a',b,b'$ which is broken up into four paths between the four of these points.
If one of these is a path $\delta$ between $a$ and $b$ then another must be a path
$\delta'$ between $a'$ and $b'$. In this case, $G$ itself is planar. If not,
there are paths from $a$ to $a'$ and $b'$ in the boundary,
and we can therefore flip $\rho_{G_0}$
along the $x$ and $y$ axes so that $a$ lands in $\{(x,y): x>0, y>0\}$ and thus
$a'$ lands in $\{(x,y): x<0, y>0\}$. By scaling,
we may assume the image of $\rho_{G_0}$ lies in $[-1,1]^2$. We may then draw
edges between $\rho_{G_0}(a)$ and $(0,1)$, $\rho_{G_0}(a')$ and $(-1,0)$,
$\rho_{G_0}(b)$ and $(0,-1)$, as well as $\rho_{G_0}(b')$ and $(1,0)$, none of
which intersect each other or any other point of $\rho_{G_0}(G_0)$.

These edges induce an embedding of $G$ into $\RR^2/2\ZZ^2$ by identifying
opposite edges of $[-1,1]^2$. We therefore have shown that $G$ is toroidal
in all cases.\end{proof}

One could imagine extending this to the case where $G/\alpha$ has genus $g$,
but that would depend on finding a sequence of points interchanged by $\alpha$
which sequentially lie on common faces. This fails however, as we see in the
following example where $G/\alpha$ has genus $2$.

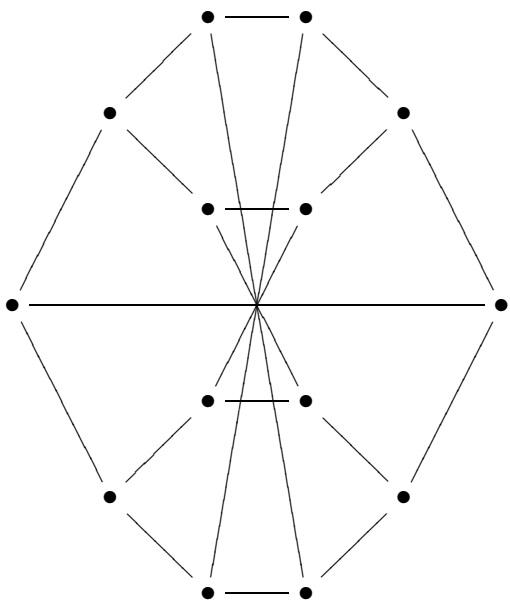
\[
\xymatrix{
 & & \bullet \ar@{-}[r] \ar@{-}[rdddddd]& \bullet & & \\
 & \bullet \ar@{-}[ru]\ar@{-}[rd] & & & \bullet\ar@{-}[lu]\ar@{-}[ld] & \\
 & & \bullet\ar@{-}[r]\ar@{-}[rdd] & \bullet & & \\
\bullet \ar@{-}[ruu]\ar@{-}[rdd]\ar@{-}[rrrrr] & & & & & \bullet\ar@{-}[luu]\ar@{-}[ldd] \\
 & & \bullet\ar@{-}[r]\ar@{-}[ruu] & \bullet & & \\
 & \bullet\ar@{-}[ru]\ar@{-}[rd] & & & \bullet\ar@{-}[lu]\ar@{-}[ld] & \\
 & & \bullet\ar@{-}[r]\ar@{-}[ruuuuuu] & \bullet & &
}
\]

Nonetheless we note that this graph does indeed admit an embedding into a genus
two surface! In particular, we get slightly lucky in that the above method
shows how to embed this graph into the connected sum of two tori, albeit in a
way that does not obviously generalize.
Indeed, we do not know of an example of a graph with a mixing
involution $\alpha$ to a genus $g$ graph which does not already embed into a
genus $g$ orientable  surface.
Sometimes as well, this construction is not optimal because
different lifts of an edge need not cross: $K_5$ admits an essentially unique
involution whose quotient has genus 2, but is well-known to be toroidal.

We conclude by noting that although our criterion for being toroidal has
\emph{something} to do with gonality, there is more that goes into the
orientable genus than the gonality.

\begin{Lemma} There are trigonal graphs of all possible orientable genera.
Moreover, there are $d$-gonal graphs which are either planar or
of all possible orientable genera
$\ge (\frac{d}{2} -1)^2$ whenever $d\not \equiv  2\bmod 4$.
\end{Lemma}

\begin{proof}
First we note that there are $d$-gonal planar graphs for all $d$ - simply take
$n\ge d$ and note that the $d\times n$ grid graph has gonality $d$
\cite[Example 3.3]{Treewidth}.

Then note that for $3\le d\le n$, the complete bipartite graph has orientable
genus $\left\lceil \frac{(d-2)(n-2)}{4}\right\rceil$. If $d$ is not $2\bmod 4$
 then this can be any integral value at least
$(\frac{d}{2} -1)^2$. On one hand, there is a clear degree $d$ harmonic
map from $K_{d,n}$ to a tree given by simply identifying the vertices in the
size $d$ subset. Therefore the gonality of $K_{d,n}$ is at most $d$. On the
other hand, the treewidth of $K_{d,n}$ is $d$, so this is a lower bound for
gonality \cite{Treewidth}, and we find that $K_{d,n}$ is $d$-gonal.
\end{proof}

The use of the complete bipartite graph above was suggested by Spencer Backman
and we thank him for the suggestion.
We conclude by noting that in the above examples, gonality, stable gonality, and
treewidth all coincide. It is conjectured for the hypercube graph $Q_n$ that
there is a gap between the two which increases along with $n$
\cite[\S 3]{Treewidth}.
In that case, the orientable genus is large and the conjectural least degree
map to a tree is given by successive quotients by involutions $Q_n \to Q_{n-1}$.
It would be interesting to find other infinite families of graphs with gaps
between gonality and treewidth and see if those also have large orientable
genus. It also still seems reasonable to wonder about
a connection between the orientable genus of a graph and the spectrum of its
Laplacian. After all, the spectrum of the $d\times n$ grid graph is
very limited \cite{Eigen}: the eigenvalues can only be
$$\lambda_{j,k} = 4\sin^2\left(\frac{j\pi}{2n}\right)
+ 4\sin^2\left(\frac{k\pi}{2d}\right).$$

In particular, the spectral lower bound on gonality \cite[Theorem C]{CKK} for
this example tends to 0 as $n\to\infty$.

\end{document}